\documentclass[12pt,reqno]{amsart}

\usepackage[margin=1in]{geometry}  
\usepackage{graphicx}              
\usepackage{amsmath}               
\usepackage{amsfonts}              
\usepackage{amsthm}                
\usepackage{amssymb}
\usepackage{appendix}
\usepackage[utf8]{inputenc}
\usepackage{mathabx}

\newtheorem{thm}{Theorem}[section]
\newtheorem{lem}[thm]{Lemma}
\newtheorem{prop}[thm]{Proposition}

\newcommand{\bd}[1]{\mathbf{#1}}  
\newcommand{\CC}{\mathbb{C}}

\newcommand{\mat}[1]{\left(\begin{matrix} #1 \end{matrix} \right)}  
\newcommand{\al}[1]{\begin{align}#1\end{align}}
\newcommand{\aln}[1]{\begin{align*}#1\end{align*}}
\newcommand{\tf}{\mathfrak{t}}

\newcommand{\Tf}{\mathfrak{T}}

\newcommand{\ff}{\mathfrak{f}}
\newcommand{\rr}{\mathcal{R}_N}
\newcommand{\hr}{\widehat{\mathcal{R}}_N}
\newcommand{\lp}{\left(}
\newcommand{\rp}{\right)}
\newcommand{\Mod}[1]{\ (\text{mod}\ #1)}

\begin{document}

\title{On representation of Integers from Thin Subgroups of $SL(2,\mathbb{Z})$ with Parabolics}

\author{Xin Zhang }
\address{University of Illinois at Urbana-Champaign, Department of Mathematics} \email{xz87@illinois.edu}
\maketitle

\begin{abstract} 
Let $\Lambda<SL(2,\mathbb{Z})$ be a finitely generated, non-elementary Fuchsian group of the second kind, and $\bd{v},\bd{w}$ be two primitive vectors in $\mathbb{Z}^2-\bd{0}$.  We consider the set $\mathcal{S}=\{\left\langle \bd{v}\gamma,\bd{w}\right\rangle_{\mathbb{R}^2}:\gamma\in\Lambda\}$, where $\left\langle\cdot,\cdot\right\rangle_{\mathbb{R}^2}$ is the standard inner product in $\mathbb{R}^2$.  Using Hardy-Littlewood circle method and some infinite co-volume lattice point counting techniques developed by Bourgain, Kontorovich and Sarnak, together with Gamburd's 5/6 spectral gap, we show that if $\Lambda$ has parabolic elements, and the critical exponent $\delta$ of $\Lambda$ exceeds 0.998317, then a density-one subset of all admissible integers (i.e. integers passing all local obstructions) are actually in $\mathcal{S}$, with a power savings on the size of the exceptional set (i.e. the set of admissible integers failing to appear in $\mathcal{S}$).  This supplements a result of Bourgain-Kontorovich, which proves a density-one statement for the case when $\Lambda$ is free, finitely generated, has no parabolics and has critical exponent $\delta>0.999950$.    
\end{abstract}

\section{introduction}

In the last few years there has been a rapidly increasing interest in studying integers from thin group orbits.  Bourgain, Gamburd and Sarnak \cite{BGS06},\cite{BGS10} introduced the notion ``affine linear sieve'', which stresses the application of ideas from classical sieve methods to finding integers with few prime factors from orbits of groups of affine linear maps.  So far the affine linear sieve has successfully produced almost primes in great generality, but just like most other sieves, the affine linear sieve has not been able to produce primes.  \par

Nevertheless, in some sporadic situations, one can do better with other methods.  For example, in \cite{BK10} Bourgain and Kontorovich studied the following question:   Let $\Lambda\subset SL(2,\mathbb{Z})$ be a free, finitely generated Fuchsian group of the second kind with no parabolic elements, and let $\bd{v}=(v_1,v_2),\bd{w}=(w_1,w_2)$ be two primitive vectors in $\mathbb{Z}^2-\bd{0}$.  Then what integers can appear in the set $\mathcal{S}=\left\langle \bd{v}\Lambda,\bd{w}\right\rangle_{\mathbb{R}^2}=\{\left\langle \bd{v}\gamma,\bd{w}\right\rangle_{\mathbb{R}^2}:\gamma\in\Lambda\}$, where $\left\langle\cdot,\cdot\right\rangle_{\mathbb{R}^2}$ is the standard inner product of $\mathbb{R}^{2}$? \par

Using Hardy-Littlewood circle method, Bourgain and Kontorovich successfully showed that as long as $\delta(\Gamma)$, the critical exponent of $\Gamma$ exceeds 0.9999493550,  the integers appearing in the set $\mathcal{S}$ has full density in the set of all \emph{admissible} numbers (defined at \eqref{0537}), with a power savings in the asymptotic convergence rate.  A direct but very interesting corollary is that there are infinitely many primes in $\mathcal{S}$.  See \cite{BK10}, \cite{BK14}, \cite{BK14i}, \cite{Hu15}, \cite{Zh14} and \cite{FSZ17},  for other works on the local-global theorems for thin group orbits.  \par

The work \cite{BK10} left the case when $\Lambda$ has parabolic elements open.  The purpose of this paper is to extend the result of  \cite{BK10} to cover this case.  Inspired by \cite{BK10}, we will also use circle method,  and the infinite-covolume lattice point counting technique developed by Bourgain, Kontorovich and Sarnak \cite{BKS10}.   However, we have to use a different setup of the ensemble. It doesn't seem likely that the setup in \cite{BK10} (or its variants) can cover the parabolic case.  Roughly speaking, this is because the existence of parabolic elements would create high multiplicity among the input vectors, which would detriment the minor arc analysis.  Instead of using the setup of  \cite{BK10}, we exploit the parabolic structure in our approach.  The existence of a parabolic subgroup implies that the representation set $\mathcal S$ contains lots of arithmetic progressions.  A heuristic is that sufficiently many arithmetic progressions should cover a density-one subset of $\mathcal S$.

\subsection{Statement of the main theorem}

Let $\Lambda<SL(2,\mathbb{Z})$ be a finitely generated, non-elementary Fuchsian group of the second kind containing parabolic elements.  Let $\bd{v}=( v_1,v_2 ),\bd{w}=( w_1,w_2 )$ be two primitive vectors in $\mathbb{Z}^{2}$ (i.e. gcd$(v_1,v_2)$=gcd$(w_1,w_2)$=1).  The purpose of this paper is to study the set $\mathcal{S}=\left\{\left\langle\bd{v}\gamma,\bd{w}\right\rangle_{\mathbb{R}^2} |  \gamma\in\Lambda \right\}$.  The two vectors $\bd{v}$ and $\bd{w}$ are fixed throughout this paper. \par

Let $\mathcal{A}$ be the set of \emph{admissible} integers by $\mathcal{S}$, i.e. 
\al{\label{0537}\mathcal{A}=\{n\in\mathbb{Z}| n\in \mathcal{S}(\text{mod } q)\} \text{ for all $q\in\mathbb{N}$.}
}
  
Since $\Gamma$ has a parabolic subgroup, the critical exponent $\delta(\Gamma)$ is greater than 1/2 \cite{Be66}, thus $\Lambda$ is a Zariski-dense subgroup of $SL(2,\mathbb{Z})$.    A more-or-less direct corollary from the strong approximation property for $\Lambda$ \cite{We84} is that the local obstruction of $\mathcal{A}$ is completely determined by some positive integer $\mathcal{Z}$: 
\begin{prop}\label{local}There exists a positive integer $\mathcal{Z}$, such that 
$$n\in\mathcal{A}\Longleftrightarrow n\in\mathcal{A}(\textrm{mod } \mathcal{Z}).$$
\end{prop}
Thus if we let $\mathcal{A}(N)=\mathcal{A}\cap[-N,N]$, then 
\aln{\#\mathcal{A}(N)=2cN+O(1),}
where \al{\label{1221}
 c=\frac{1}{\mathcal{Z}}\cdot\# \{\text{admissible congruence classes mod }\mathcal{Z}\}.}

We have an obvious inclusion $\mathcal{S}\subset \mathcal{A}$, but on the other hand a local-global principle also predicts that $\mathcal{A}\subset \mathcal{S}$.  The purpose of this paper is to prove an asymptotic local-global principle in the following sense, which is our main theorem:

\begin{thm}\label{mainthm}
Let $\Lambda \subset SL(2,\mathbb{Z})$ be a finitely generated Fuchsian group with parabolic elements.  Let $\mathcal{S}(N)=\mathcal{S}\cap [-N,N]$.  
Then there exists $5/6<\delta_0<1$, such that if $\delta=\delta(\Lambda)>\delta_0$, we have
\al{\label{0256}{\# \mathcal{S}(N)=2cN+O(N^{1-\eta}).}}
for some $\eta>0$, where $c$ is given in \eqref{1221}.  
One choice for $\delta_0$ can be $\frac{593}{594}=0.998316498\cdots$.
\end{thm}

Thus under the assumption of Theorem \ref{mainthm}, most admissible integers are represented, with a power saving on the size of the exceptional set.\\

\subsection{A closer look at local-global principle}

Returning to the setting of Theorem \ref{mainthm},  suppose $\Lambda$ has a parabolic subgroup, say $\Gamma_\infty=\left\{\mat{1&n\\0&1}:n\in\mathbb{Z}\right\}$.  Then $\left\langle \bd{v}\mat{1&n\\0&1},\bd{w}\right\rangle_{\mathbb{R}^2}=v_1w_1+v_2w_2+v_1w_2 n$, viewed as a linear form of $n$, already gives a positive density subset of $\mathbb{Z}$ when $n$ goes over all integers.    Moreover, we can obtain other such linear forms by precomposing $\bd{v}$ with some element $\gamma\in\Lambda$.  One might wonder if this problem is elementary after all: can one cleverly find finitely many such forms which obtain all admissible integers, thus the whole local-global principle is proved? 

Indeed,  let's take a look at the Lubotzky 3-group as an example.  Let $\Lambda=\left\langle \mat{1&3\\0&1},\mat{1&0\\3&1}\right\rangle$ and $\bd{v}=\bd{w}=(0,1)$, then we are looking at the 2-2 entries of $\Lambda$.  Write $\gamma=\mat{a_\gamma&b_\gamma\\c_\gamma&d_\gamma}$.  It is not hard to see that the local obstruction for $\mathcal{S}$ appears at 9: if $\gamma\in\Lambda$, then $d_\gamma\equiv 1(\text{mod } 9)$.  On the other hand, a single linear form $\left\langle(0,1)\mat{1&0\\3&1}\mat{1&3n\\0&1},(0,1)\right\rangle_{\mathbb{R}^2}=9n+1$ has already produced all admissible integers.  

However, if we set $\bd{v}=(0,1),\bd{w}=(7,5)$, then the local obstruction appears at 3 (if $\gamma\in\Lambda$, then $7c_\gamma+5d_\gamma\equiv2 \Mod{3}$), but the numerical evidence shows that among all admissible integers up to 200000, 593 are missing to be represented, and no admissible integer greater than 200000 is found missing ever since.  If one can find finitely many linear forms to cover all represented integers,  then the admissible set $\mathcal{A}$ is the same as represented set $\mathcal{S}$. In other words all admissible integers are represented.  Thus as long as we have at least one admissible integer missing, we would need infinitely many linear progressions to cover all represented integers.  In this case, any finite union of arithmetic progressions $\mathcal S'$ can not cover a density one subset of $\mathcal A$. This can be seen as follows:  both $\mathcal S'$ and $\mathcal A$ can be written as finite unions $U_1$ and $U_2$ of congruence classes mod $\mathcal Z_0$, for some common modulus $\mathcal Z_0$. Since $\mathcal S'$ is a proper subset of $\mathcal A$, $U_1$ is a proper subset of $U_2$.  As a result, $\mathcal S'$ covers a positive-density, but not a density-one subset of $\mathcal A$. 

Next, to show that Theorem \ref{mainthm} indeed covers nontrivial examples, in Theorem \ref{example} we prove the existence of a family of groups containing parabolics whose critical exponents can get arbitrarily close to 1 and whose 2-2 entries can miss arbitrarily finitely many admissible integers.  In \cite{BK10} Bourgain-Kontrovich proved similar properties for certain subgroups of a parabolic free group $\Gamma'(2)$.  We add an extra parabolic element to these groups in our construction.

Let $\Gamma(q)$ be the classical principle congruence subgroup of level $q$, i.e. $\Gamma(q)=\{\gamma\in SL(2,\mathbb{Z}):\gamma\equiv I (\text{mod } q)\}$ (we extend this definition to other algebraic groups in an obvious way).  \\

It is well known that $\Gamma(2)$ is a free group generated by 
\al{A=\mat{1&2\\0&1}, B=\mat{1&0\\2&1}.}

Let $\Gamma'=[\Gamma(2),\Gamma(2)]$ be the commutator group of $\Gamma(2)$.  Then all elements in $\Gamma'$ are of the form
\al{\label{0110}A^{m_1}B^{n_1}A^{m_2}B^{n_2}\cdots A^{m_k}B^{n_k}  }
with $\sum_i m_i=\sum_i n_i=0$.\\

We will show:

\begin{thm}\label{example} For any $M>0$ and $0<\delta_1<1$, there exists a finitely generated subgroup $\Lambda=\Lambda_{M,\delta_1}$ of $\left\langle\Gamma'(2),\mat{1&2\\0&1}\right\rangle$ such that: 
\begin{enumerate}
\item $\delta(\Lambda)>\delta_1$;
\item $\Lambda$ contains $\mat{1&2\\0&1}$;
\item Let $\mathcal{A}$ be the set of integers admissible by $\left\langle (0,1)\Lambda, (0,1)\right\rangle_{\mathbb{R}^2}$, then 
$$\#(\mathcal{A}-\left\langle (0,1)\Lambda, (0,1)\right\rangle_{\mathbb{R}^2})>M.$$
\end{enumerate}
\end{thm}

\begin{proof}
The group $\Gamma'$ is an infinite index, Zariski dense subgroup of $SL(2,\mathbb{Z})$, with critical exponent being 1 (followed by a counting argument in \cite{Ep87}).  It is straightforward to see that an integer can appear at the 2-2 entry of some matrix in $\Gamma'$ if and only if this integer is odd. \\

For any matrix in $\Gamma(2)$, composing by a multiple of $A$ on the left and a multiple of $B$ on the right, one can always bring a matrix to a unique matrix of the form $\mat{a&b\\c&d}$ with the 2-2 entry $d$ unchanged and $|b|,|c|< |d|$, and we call such matrices \emph{primitive}.\\

Thus, for any odd $d_0$, there exist finitely many primitive matrices $M_{d_0,1},M_{d_0,2},\cdots, M_{d_0,n(d_0)}$ in $\Gamma(2)$ with the 2-2 entry being $d_0$.  Then all matrices with 2-2 entry $d_0$ can be expressed by 
\al{\cup_{i=1}^{n(d_0)}\cup_{m=-\infty}^{\infty}\cup_{n=-\infty}^{\infty} A^m M_{d_0,i} B^n}
Each $M_{d_0,i}$ might not be in $\Gamma'$, but from \eqref{0110} there's a unique way to left multiplying a multiple of $A$ and right multiplying a multiple of $B$, to get a matrix $\widetilde{M_{d_0,i}}\in\Gamma'$.\\

Since the critical exponent of $\Gamma'$ is 1, by Corollary 6 of \cite{Su79}, we can find a finitely generated subgroup $\Lambda'$ of $\Gamma'$ such that $\delta(\Lambda')>\delta_1$ for any $0<\delta_1<1$.  If $\delta_1> 1/2$, then the Zariski closure of $\Lambda'$ is $SL(2,\mathbb R)$, and it is known that $\Lambda'$ satisfes a strong approximation property \cite{We84}.  In other words, there exists an integer $\mathcal{B}_0$, such that $\Lambda'(\mathcal{B}_0)$, the principal congruence subgroup of 
$\Lambda'$ of level $\mathcal{B}_0$, satisfies 
\aln{\Lambda'(\mathcal{B}_0)/\Lambda'(\mathcal{B}_0)(q) \cong \Gamma(\mathcal{B}_0)/\Gamma(\mathcal{B}_0)(q)
}
for any $q\in\mathbb{N}$.

Now take $N$ distinct odd primes $p_1,p_2,\cdots,p_N$ which are congruent to 1 mod $\mathcal{B}_0$.  Let $\Lambda''=\left\langle \Lambda', A  \right\rangle$.  Then the set of all matrices in $\Lambda''$ that have $p_i$ as the 2-2 entry in $\Lambda_1$ can be expressed as 
\al{\label{0859}   \{A^m\widetilde{M_{p_i,j}}:\widetilde{M_{p_i,j}}\in\Lambda',1\leq j\leq n(p_i),-\infty<m<\infty\}
}

It is noted that there are only finitely many different 2-1 entries in \eqref{0859}.  Let $K_{p_i}$ denote the maximum of the absolute values of the 2-1 entries from \eqref{0859}. \\

Now we pick up a large prime number $P$ such that $P>\max_{\{1\leq i\leq N\}} K_{p_i}$.  Let $\Lambda=\Lambda^{''}\cap \Gamma_0 (P) $, where  
$$\Gamma_0(P)=\left\{\mat{a&b\\c&d}\in SL(2,\mathbb{Z})|c\equiv 0\Mod{P}\right\}.$$

Then no matrix in $\Lambda$ has the 2-2 entry any of the $p_i$'s.  But on the other hand it is straightforward to check that each of the $p_i$'s passes every local obstruction.
\end{proof}

We hope we have illustrated effectively the subtlety of our problem.  While being trivial in certain cases, the problem can become quite nontrivial in some other cases.\\

\noindent
$\bd{Our\hspace{1.5mm} method \hspace{1.5mm}for\hspace{1.5mm} proof\hspace{1.5mm} of \hspace{1.5mm}Theorem\hspace{1.5mm} \ref{mainthm}}$:  Just like the predecessor of this paper \cite{BK11}, we will use Hardy-Littlewood circle method combined with the orbital counting technique developed in \cite{BKS10} to prove Theorem \ref{mainthm}.  Let $N$ be the main growing parameters, and write $N=TX$, where $T$ is thought to be a small power of $N$.  Our ensemble is a set $B_T$ consisting of elements $\gamma\in \Lambda$ pointing at some specific angles with matrix norm bounded by $T$.  Each $\gamma$ corresponds to a linear form $\ff_{\gamma}$.  We then let $x$ range over some interval comparable to $X$ (in reality we weight $x$ by (the dilation of) some compactly-supported smooth function $\psi$).  We then define a function $\mathcal{R}_N$ on $\mathbb{Z}$ such that $\mathcal{R}_N(n)$ roughly captures the number of times $n$ is represented as $\ff_\gamma (x)$ and has the property that $\mathcal R_N(n)>0$ if and only if $n$ is represented.  By analyzing the Fourier transform of $\mathcal{R}_N$, we show that $\mathcal{R}_N (n)>0$ for almost every admissible $n$, thus most admissible integers are represented.   \\

\noindent
$\bd{Remark}$:  One might ask if one can improve Theorem \ref{mainthm}, say, by impriving the errior term to be $O(1)$, or by relaxing the restriction that $\delta$ is sufficiently close to 1.  In both cases, it seems that significantly new ideas are needed.  We have tried to introduce two and more variables (corresponding to multiple copies of parabolic subgroups), but just like the one-variable method in our paper, they lead to a same obstacle in some part of the minor arc analysis, which requires the critical exponent to be sufficiently close to 1 in order to get cancellation in average.    \\

\noindent
$\bd{Plan\hspace{1.5mm} of \hspace{1.5mm}the\hspace{1.5mm} paper}$:
In Sec.\ref{setup}, we describe the setup of the ensemble for the circle method.  In Sec.\ref{estimate}, we state some counting results on infinite co-volume lattices of $SL(2,\mathbb{R})$ developed by Bourgain-Kontorovich-Sarnak \cite{BKS10}.  Sec.\ref{major} and Sec.\ref{minor} are devoted to major and minor arc analysis, respectively.  \\  

\noindent
$\bd{Notations}$: The Greek letters $\epsilon,\eta$ are small positive numbers, and the letter $L$ is a large positive number, all of which appear in several contexts. We assume that each time when $\epsilon, \eta, L$ appear, we let $\epsilon,\eta$ not only satisfy the current claim, but also satisfy the claims in all previous contexts.  The symbol $\sum_{a(q)}$ denotes a summation over all residues mod $q$ and the symbol $\sum_{a(q)}'$ denotes a summation over all residues $a$ mod $q$ with $(a,q)=1$. The relation $f\ll g$ is synonymous with $f = O(g)$, and $f\asymp g$ means $f\ll g$ and $g\ll f$.  Without further specifying, all the implied constants depend at most on the group $\Lambda$, the vectors $\bd{v}$ and $\bd{w}$ and the quantities $\epsilon, \eta, L$. \\

\noindent
$\bd{Acknowledgement}$:
The author is grateful for Professor Alex Kontorovich for proposing the question, and his many detailed comments which lead to an improvement of an earlier version of this paper.  Thanks are also given to Junxian Li and anonymous referees for numerous corrections for preliminary versions of this paper.  

\section{\label{setup}Setup of the circle method}

Recall that $\Lambda$ is a finitely generated, non-elementary subgroup of $SL(2,\mathbb{Z})$ containing a parabolic subgroup.  Without loss of generality we can assume this parabolic subgroup is $\Lambda_\infty=\left\{\mat{1&Jn\\0&1}:n\in\mathbb{Z}\right\}$, because we can conjugate any parabolic subgroup to be such a group.  Let $\bd{v}=(v_1,v_2)$ and $\bd{w}=(w_1,w_2)$ be two primitive vectors in $\mathbb{Z}^2-\bd{0}$.  We assume that $v_1,w_2\neq 0$; otherwise we can always precompose $\bd{v},\bd{w}$ with some group elements from $\Lambda$ to have this property since $\Lambda$ is non-elementary.  \par

Observe that the set $\left\langle \bd{v}\cdot\Lambda_{\infty}, \bd{w}\right\rangle_{\mathbb{R}^2}=\{v_1w_1+w_2v_2+Jw_2v_1n| n\in\mathbb{Z}\}$ already gives a positive density subset of $\mathcal{S}$.  Moreover, other sets of such linear progressions can be obtained by precomposing $\bd{v}\cdot\Gamma_\infty$ with some $\gamma=\mat{a_\gamma&b_\gamma\\c_\gamma&d_\gamma}\in\Lambda$.\\

Write $\ff_\gamma(n)=A_{\gamma}n+B_\gamma$, where 
\al{\label{0119}&A_\gamma=v_1(c_\gamma w_1+ d_\gamma w_2)J\\&B_\gamma= v_1(a_\gamma w_1+b_\gamma w_2)+v_2(c_\gamma w_1+ d_\gamma w_2).}

Then we have $\left\langle \bd{v}\cdot\Gamma_\infty\cdot\gamma,  \bd{w}\right\rangle_{\mathbb{R}^2}=\{A_\gamma n +B_\gamma | n\in\mathbb{Z}\}$.   \par
 
Let $G=SL(2,\mathbb{R})$.  We use the standard matrix norm $\Vert\cdot\Vert$ on $G$ given by $$\left\Vert \mat{a&b\\c&d}\right\Vert=\sqrt{a^2+b^2+c^2+d^2}$$

Let $N$ be the main growing parameter.  We introduce two parameters $T$ and $X$ with $TX=N$.  The parameter $T$ is a small power of $N$.  In fact $T$ could be any $N^{\alpha}$ with $\alpha\in(0,1/2)$.  We define the following homogeneous growing set $B_T$:\\

$$B_T:=\left\{\gamma\in \Lambda\Big\vert \Vert\gamma\Vert<T, |A_{\gamma}|\geq{\frac{T}{100}}\right\}.$$

Let $\psi $ be a smooth non-negative function supported on $(0.5,2.5)$ and $\psi \geq 1$ on $[1,2]$. \\ 

Now we are ready to define our representation function $\mathcal{R}_N$:
\al{
\mathcal{R}_N(n)=\sum_{\gamma\in B_T}\sum_{x\in\mathbb{Z}} \psi \left(\frac{x}{X}\right)\bd{1}\{\ff_\gamma(x)=n\}.
}     

Note that $\mathcal{R}_N(n)>0$ implies $n$ is in the set $\mathcal{S}$.\\

The Fourier transform of $\mathcal{R}_N$ is
\al{\label{rhat}\widehat{\mathcal{R}}_N (\theta)= \sum_{\gamma\in B_T}\sum_{x\in\mathbb{Z}} \psi\left (\frac{x}{X}\right)e(\ff_\gamma(x)\theta).}

One can recover $\mathcal{R}_N$ from $\widehat{\mathcal{R}}_N$ by the following Fourier inversion formula:

\al{\label{0322}\mathcal{R}_N(n)=\int_{0}^1\widehat{R}_N(\theta)e(-n\theta)d\theta.}

From Dirichlet's approximation theorem, given any positive integer $M$, and for every real number $\theta\in[0,1)$, there exist coprime integers $a,q$ such that $1\leq q \leq M$ and $\left\vert\theta-\frac{a}{q}\right\vert <\frac{1}{qM}$.  Write $\theta=\frac{a}{q}+\beta$.  The integer $M$ is also called the depth of approximation.  In our context we set $$M=T^{1+\epsilon_0},$$
where $\epsilon_0$ is a fixed and an arbitrarily small positive quantity. \par
 A general philosophy for circle method is that most contribution to the integral \eqref{0322} should come from the neighborhoods of rationals with small denominators, and we call such neighborhoods \emph{major arcs}.  We introduce two parameters $Q_0,K_0$ such that the major arcs corresponds to $q\leq Q_0,\beta\leq\frac{K_0}{N}$.  Both $Q_0$ and $K_0$ are small powers of $N$ and are determined at \eqref{0933}.   \\

We define the following hat function
\aln{\tf:=\text{min}(1+x,1-x)^{+}}
whose Fourier transform is
\aln{\hat{\tf}(y)=\left(\frac{\text{sin}(\pi y)}{\pi y}\right)^2.}
In particular, $\hat{\tf}$ is a nonnegative function. \\

From $\tf$, we construct a spike function $\mathfrak{T}$ which captures the major arcs:
\al{\label{spike}\mathfrak{T}(\theta):=\sum_{q\leq Q_0}\sum_{a(q)}{}^{'}\sum_{m\in\mathbb{Z}}\tf\left(\frac{N}{K_0}\left(\theta+m-\frac{a}{q}\right)\right).}\\

We then define our ``main" term to be 
\al{\label{mainterm1}\mathcal{M}_N(n):=\int_0^1\mathfrak{T}(\theta)\widehat{R}_N(\theta)e(-n\theta)d\theta}
and the ``error" term to be
\al{\label{errorterm}\mathcal{E}_N(n):=\int_0^1(1-\Tf(\theta))\widehat{R}_N(\theta)e(-n\theta)d\theta.}
\\
Theorem 1.5 of \cite{BKS10} implies that $\#B_T\asymp T^{2\delta}$.  Therefore, the total input $\widehat{\mathcal{R}}_N(0) \asymp T^{2\delta}X$.  We expect the total mass is equidistributed among all admissible $n\asymp N$.  Indeed, in Section \eqref{major}, we will show that all admissible $n\in[-N,-N/2]\cup[N/2,N]$,  we have $\mathcal{M}_N(n) \gg T^{2\delta-1}$, ignoring log factors:
\begin{thm} \label{mainterm}For $N/2<|n|<N$, the main term $\mathcal{M}_N(n)$ is 
\begin{displaymath}
 \begin{array}{ll}
\gg \frac{1}{\log\log(10+|n|)}T^{2\delta-1} & \textrm{if $n\in\mathcal{A}$},\\
\ll T^{2\delta-1-\epsilon}\text{ for some $\epsilon$}>0 & \textrm{otherwise}.
\end{array} 
\end{displaymath}
\end{thm}

We are not able to give a satisfactory individual bound for $\mathcal{E}_N(n)$ for each $n$, which would improve the error term of \eqref{0256} to be $O(1)$.  However, we are able to give an $l^2$ bound for $\mathcal{E}_N$, which shows $\mathcal{E}_N$ is small in average:
\begin{thm}\label{minorbound}
There exists a positive $\eta$ such that 
$$\sum_{n\in\mathbb{Z}}\mathcal{E}_N(n)^2\ll T^{4\delta-2}N^{1-\eta}.$$
\end{thm}

Assuming Theorem \ref{mainterm} and Theorem \ref{minorbound}, Theorem \ref{mainthm} follows from a standard argument:
\begin{proof}[Proof of Theorem \ref{mainthm} assuming Theorem \ref{mainterm} and Theorem \ref{minorbound}]

Let $E(N)$ be the set of admissible integers in $[-N,-N/2]\cup[N/2,N]$ not in $\mathcal{S}$, then for each $n\in E(N)$, 
$$\rr(n)=0$$ and $$\mathcal{E}_N(n)=\mathcal{M}_N(n)\gg \frac{1}{\log\log(|n|+10)}T^{2\delta-1}.$$   Therefore,
$$\#E(N)T^{4\delta-2}/(\log\log(|n|+10))^2\ll \sum_{n\in E(N)}\vert\mathcal{E}_N(n)\vert^2\ll \sum_{n\in \mathbb{Z}}\vert\mathcal{E}_N(n)\vert^2\ll T^{4\delta-2}N^{1-\eta},$$
which leads to
$$\#E(N)\ll N^{1-\eta},$$
ignoring the log factors. 
\end{proof}

\section{\label{estimate}Several Orbital Counting Estimates}
In this section we state some orbital counting estimates which are used in the sequel.  We carry over all previous notations. \par

Let $G=SL(2,\mathbb{R})$ and $\Lambda< SL(2,\mathbb{Z})$ be a finitely generated, non-elementary Fuchian group of the second kind containing parabolics, and let $\delta$ be the Hausdorff dimension of the limit set of $\Lambda$.  We assume $\delta>5/6$.  \par

We require the following Sobolev-type norm.  Let $\{X_i\}$ be an orthonormal basis for the Lie algebra $\mathfrak{g}=\mathfrak{sl}(2,\mathbb{R})$.  Each vector $X_i$ is then extended to a left $G$-invariant vector field on $G$, for which we still denote by $X_i$.  Then the norm $\mathcal{S}_{\infty,T}$ is defined by\\

$$\mathcal{S}_{\infty,T}f=\sup_{i}\sup_{\Vert g\Vert< T} |X_i f(g)|, $$ for $f\in C^{\infty}(G)$.  That is,  $\mathcal{S}_{\infty,T}$ is the supreme norm of the first order derivatives of $f$ in the ball of radius $T$ about the identity.  \par

Let $\Lambda(q)$ be the principal congruence subgroup of $\Lambda$, i.e. $\Lambda(q)=\{\gamma\in\Lambda| \gamma\equiv I (\text{mod } q)\}$, and let $\Lambda_q$ be the set $\Lambda/\Lambda(q)$.\par

The following two orbital-counting theorems are due to Bourgain-Kontorovich-Sarnak \cite{BKS10}.
 
\begin{thm} \label{count1} Let $\Lambda$ be as above. Fix any $\gamma_0\in \Lambda$ and $q>1$.  Let $f: G\rightarrow \CC$ be a smooth function with $|f|\leq 1$.  There exists a fixed ``bad" integer $\mathcal{B}$ such that for $q=q^{}{'}q^{}{''},q^{}{'}\mid \mathcal{B}$, we have 
\al{\label{0645}\sum_{\substack{\gamma\in B_T\\ \gamma\equiv \gamma_0 (\Lambda(q))}}f(\gamma)= \frac{1}{\# \Lambda_q}(\sum_{\gamma\in B_T}f(\gamma)+\mathfrak{E}_{q^{}{'}})+O((1+\mathcal{S}_{\infty,T}f)^{\frac{6}{7}}T^{\frac{6}{7}2\delta+\frac{5}{21}})}
Here $\mathfrak{E}_{q^{}{'}} \ll T^{2\delta-\alpha_0}$ for some $\alpha_0>0$.  All the implied constants are independent of $\gamma_0$ and $q$. 
\end{thm}

\begin{proof}  In \cite{BKS10} Bourgain-Kontorovich-Sarnak proved a version of this theorem with $B_T$ replaced by the expanding set $\{\gamma\in \Lambda: \Vert\gamma\Vert<T\}$.  The same proof goes through for $B_T$ as well. 
 \end{proof}

\begin{thm} \label{count2}Let $\bd{v},\bd{w}\in \mathbb{Z}^2$ and assume $N/2<|n|<N$ and $X\leq x\leq 2X$.  Then we have  
$$\sum_{\gamma\in B_T}\bd{1}\left\{|\left\langle \bd{v}\mat{1&x\\0&1}\cdot\gamma,\bd{w}\right\rangle_{\mathbb{R}^2}-n|\leq \frac{N}{2K_0} \right\}\gg \frac{T^{2\delta}}{K_0}+O(T^{\frac{3}{4}+2\delta\frac{1}{4}}{(\log T)^{\frac{1}{4}}})$$
\end{thm}
\begin{proof}This is application of Theorem 1.15 in \cite{BKS10}.
\end{proof}

Theorem \ref{count2} essentially counts points of $\Lambda$ pointing at some shrinking angles. Such a counting is valid, as long as the shrinking rate, measured by $K_0$) is slow enough (depending on the spectral gap). 

\noindent
$\bd{Remark}$:  Theorem \ref{count1} has been extended vastly to the setting of $SO(n,1)$ by Mohammadi-Oh \cite{MO15}.

\section{\label{major}Major Arc Analysis}
The purpose of this section is to prove Theorem \ref{mainterm}.  The main player in our analysis is the $B_T$-sum of $\widehat{R}_N(\theta)$ (see \eqref{rhat}), that is, we fix $x$ in \eqref{rhat}, and we apply Theorem \ref{count1} to the $B_T$-sum to get the desired estimate for $\widehat{R}_N(\theta)$. \\  

Inserting \eqref{rhat} into \eqref{mainterm1}, we obtain
\al{\label{0130}\nonumber
\mathcal{M}_N(n)=&\int_0^1 \sum_{q\leq Q_0}\sum_{a(q)}{}^{'}\sum_{m\in\mathbb{Z}} \tf\left(\frac{N}{K_0}\left(\theta+m-\frac{a}{q}\right)\right)\widehat{R}_N(\theta)e(-n\theta)d\theta
\\
\nonumber=&\sum_{q\leq Q_0}\sum_{a(q)}{}^{'}\int_{\mathbb{R}}\tf\left(\frac{N}{K_0}\beta\right)\widehat{\mathcal{R}}_N(\frac{a}{q}+\beta)e(-n(\frac{a}{q}+\beta))d\beta\\
=&\sum_{x\in\mathbb{Z}} \psi (\frac{x}{X})\sum_{q\leq Q_0}\sum_{a(q)}{}^{'}\int_{\mathbb{R}}\tf\left(\frac{N}{K_0}\beta\right)\sum_{\gamma\in B_T}e\left(\left(\ff_\gamma(x)-n\right)\lp\frac{a}{q}+\beta\rp\right)d\beta
}
Now we split $\sum_{\gamma\in B_T}$ into $\Lambda(q)$-cosets  and apply Theorem \ref{count1}, and for simplicity we assume $\mathcal{B}=1$ so we don't have the $\mathfrak{E}_{q^{}{'}}$ term.  In general, if $\mathcal B\neq1$, then the $\mathfrak{E}_{q^{}{'}}$ term corresponds to a contribution coming from a ``new form" that appears at level $q'$ with spectral parameter $>\frac{5}{6}$.  By Gamburd's Theorem, there are only finitely many such forms, and the bad integer $\mathcal B$ is the least common multiple of all such $q'$.   Then we can work with $\Lambda(\mathcal B)$, the principal congruence subgroup of $\Lambda$ of level $\mathcal B$, so that we don't have the $\mathfrak{E}_{q^{}{'}}$ term in the statement of Theorem \ref{count1}, if we replace $\Lambda$ by $\Lambda (\mathcal B)$.  The set $\mathcal S$ can be written as a finite union of the sets of the form $\left\langle \bd{v}_i\Lambda(\mathcal B),\bd{w}_i\right\rangle$, and we can work with each of them and combine the results.\par

Let $f_1(g)=e((\mathfrak{f}_g(x)-n)\beta)=e\left(\left\langle \bd{v}\mat{1&x\\0&1}g,\bd{w}\right\rangle\beta\right)e(-n\beta)$, $g\in SL(2,\mathbb R)$. If $|\beta|<\frac{K_0}{N}$, then $\frac{K_0X}{N}\ll 1$, so $f_1$ has bounded derivative and we can apply Theorem \ref{count1} to $f_1$:

\al{\label{0128}
\nonumber\sum_{\gamma\in B_T}&e\left(\left(\ff_\gamma(x)-n\right)\lp\frac{a}{q}+\beta\rp\right)\\
=&\sum_{\gamma_0\in\Lambda_q}e\left(\left(\ff_{\gamma_0}(x)-n\right)\frac{a}{q}\rp\sum_{\substack{\gamma\in B_{T}\\\gamma\equiv\gamma_0(\Lambda(q))}}e\nonumber\left(\left(\ff_{\gamma}(x)-n\right)\beta\rp\\
=&\sum_{\gamma_0\in\Lambda_q}e\left(\left(\ff_{\gamma_0}(x)-n\right)\frac{a}{q}\rp\lp\frac{1}{\#\Lambda_q}\lp\sum_{\gamma\in B_T}e\left(\left(\ff_{\gamma}(x)-n\right)\beta\rp\rp+O(T^{\frac{6}{7}2\delta+\frac{5}{21}})\rp
}

Plugging \eqref{0128} back into \eqref{0130}, we have

\al{\nonumber \label{0919}
\mathcal{M}_N(n)=&\sum_{x\in\mathbb{Z}} \psi \left(\frac{x}{X}\right)\left(\sum_{q<Q_0}\frac{1}{\#\Lambda_q}\lp\sum_{a(q)}\sum_{\gamma_0\in\Lambda_q}e\left(\left(\ff_{\gamma_0}(x)-n\right)\frac{a}{q}\rp\rp\right)\\
\nonumber&\times \sum_{\gamma\in B_T}\int_{\mathbb{R}}\tf\lp\frac{N}{K_0}\beta\rp e\left(\left(\ff_{\gamma}(x)-n\right)\beta\rp d\beta +O\lp\frac{ Q_0^5 T^{\frac{6}{7}2\delta+\frac{5}{21}}K_0 X}{N}\rp \\
=&\frac{K_0}{N}\sum_{x\in\mathbb{Z}}\psi\lp\frac{x}{X}\rp\mathfrak{S}_{Q_0,x}(n)\sum_{\gamma\in B_T}\hat{t}\lp\frac{K_0}{N}(\ff_\gamma(x)-n)\rp +O\lp T^{2\delta-1}\frac{Q_0^5 K_0}{T^{\frac{2}{7}\delta-\frac{5}{21}}}\rp,
}
where \al{\label{0538}\mathfrak{S}_{Q_0,x}(n)=\sum_{q<Q_0}\frac{1}{\#\Lambda_q}\sum_{a(q)}\sum_{\gamma_0\in\Lambda_q}e\left(\left(\ff_{\gamma_0}(x)-n\right)\frac{a}{q}\rp.}
It is noted that the definition of $\mathfrak{S}_{Q_0,x}(n)$ is independent of $x$ and we can abbreviate $\mathfrak{S}_{Q_0,x}(n)$ to $\mathfrak{S}_{Q_0}(n)$. 
This is because the innermost sum 
\aln{
\sum_{\gamma_0\in\Lambda_q}e\left(\left(\ff_{\gamma_0}(x)-n\right)\frac{a}{q}\rp=\sum_{\gamma_0\in\Lambda_q}e\left(\left\langle \bd{v}\mat{1&x\\0&1}\gamma_0,\bd{w}  \right\rangle \frac{a}{q}\right)= \sum_{\gamma_0\in\Lambda_q}e\left(\left\langle \bd{v}\gamma_0,\bd{w}  \right\rangle \frac{a}{q}\right).
}
 Thus we have split each $x$-summand of $\mathcal{M}_N(n)$ into a product of a modular piece $\mathfrak{S}_{Q_0}(n)$ and an Archimedean piece $\frac{K_0}{N}\psi\lp\frac{x}{X}\rp\sum_{\gamma\in B_T}\hat{t}\lp\frac{K_0}{N}(\ff_\gamma(x)-n)\rp$.  \par

The analysis for the the modular piece $\mathfrak{S}_{Q_0}(n)$ is identical to the one in  Sec. 4.2 of \cite{BK10}, from which we have
\begin{lem}\label{modular}
\begin{displaymath}
\mathfrak{S}_{Q_0}(n) = \left\{ \begin{array}{ll}
\gg \frac{1}{\log\log(10+|n|)}& \textrm{if $n\in\mathcal{A}$}\\
\ll \frac{1}{N^{\epsilon}} \text{ for some $\epsilon$}>0 & \textrm{otherwise}.
\end{array} \right.
\end{displaymath}
\end{lem}

We apply Theorem \ref{count2} to analyze the sum $\sum_{\gamma\in B_T}$ in \eqref{0919}.  Noting that $\hat{t}(y)>0.4$ when $|y|<1/2$, for $x\in[X,2X]$, we have 

\al{\label{0507}\sum_{\gamma\in B_T}\hat{t}\lp\frac{K_0}{N}(\ff_\gamma(x^{}{'})-n)\rp \gg \frac{T^{2\delta}}{K_0}+O(T^{\frac{3}{4}+2\delta\frac{1}{4}})} 

Lemma \ref{modular}, together with \eqref{0507} thus implies Theorem \ref{mainterm} once the parameters $K_0,Q_0$ are set with the error term not overtaking the main term, or 
\al{\label{0917}&Q_0^5K_0 \ll T^{\frac{2}{7}\delta-\frac{5}{21}},\\
\label{0300} &K_0\ll T^{\frac{3}{4}2\delta-\frac{3}{4}}.
}

\section{\label{minor}Minor Arc Analysis}

The purpose of this section is to prove Theorem \ref{minorbound}, which gives an $l_2$ bound for $\mathcal{E}_N$.  The main player of our analysis is the $x-$sum of $\widehat{R}_N(\theta)$, where we can use the Poisson summation to get cancellation.

By Plancherel, proving Theorem \ref{minorbound} is the same as to prove 
\al{\label{0714}\int_0^1|1-\mathfrak{T}(\theta)|^2\vert\widehat{\mathcal{R}}_N(\theta)\vert^2 d\theta \ll  T^{4\delta-2}N^{1-\eta}.  }
To analyze \eqref{0714}, we consider the following four integrals:
\al{
I_1=&\sum_{q\leq Q_0}\sum_{a(q)}{}^{'}\int_{-\frac{K_0}{N}}^{\frac{K_0}{N}}\lp\frac{N\beta}{K_0}\rp^2\Big\vert\hr(\frac{a}{q}+\beta)\Big\vert^2d\beta,\\
I_2=&\sum_{q\leq Q_0}\sum_{a(q)}{}^{'}\int_{\frac{K_0}{N}\leq|\beta|\leq \frac{1}{qM} }\Big\vert\hr(\frac{a}{q}+\beta)\Big\vert^2d\beta,\\
I_3=&\sum_{Q_0<q\leq M}\sum_{a(q)}{}^{'}\int_{\frac{1}{N}\ll |\beta|\leq\frac{1}{qM}}\Big\vert\hr(\frac{a}{q}+\beta)\Big\vert^2d\beta ,\\
I_4=&\sum_{Q_0<q\leq M}\sum_{a(q)}{}^{'}\int_{|\beta|\ll\frac{1}{N}}\Big\vert\hr(\frac{a}{q}+\beta)\Big\vert^2d\beta ,
}
and give the bound $T^{4\delta-2}N^{1-\eta}$ for each.  Indeed, the integration intervals of $I_1,I_2,I_3,I_4$ cover the whole interval $[0,1]$ by Dirichlet's approximation theorem, and the integrants of $I_1,I_2,I_3,I_4$ dominate $|1-\mathfrak{T}(\theta)|^2\vert\widehat{\mathcal{R}}_N(\theta)\vert^2$ in the corresponding intervals.  Therefore,  \\
$$\int_0^1|1-\mathfrak{T}(\theta)|^2\vert\widehat{\mathcal{R}}_N(\theta)\vert^2 d\theta \ll I_1+I_2+I_3+I_4.$$

Recall the definition of $\hr$ from \eqref{rhat}.  We use Poisson summation for the $x$ variable to rewrite $\hr$:
\al{\label{1102}
\hr(\nonumber\frac{a}{q}+\beta)=&\sum_{\gamma\in B_T}\sum_{x\in\mathbb{Z}}\psi\left(\frac{x}{X}\right)e\left(\ff_\gamma(x)(\frac{a}{q}+\beta)\right)\\
\nonumber=&\sum_{\gamma\in B_T}\sum_{x_0(q)}e\lp\ff_\gamma(x_0)\frac{a}{q}\rp\sum_{\substack{x\in\mathbb{Z}\\x\equiv x_0(q)}}\psi\lp\frac{x}{X}\rp e\left(\ff_\gamma(x)\beta\right)\\
\nonumber=&\sum_{\gamma\in B_T}\sum_{x_0(q)}e\lp\ff_\gamma(x_0)\frac{a}{q}\rp\sum_{\substack{x\in\mathbb{Z}}}\psi\lp\frac{x_0+qx}{X}\rp e\left(\ff_\gamma(x_0+qx)\beta\right)\\
\nonumber=&\sum_{\gamma\in B_T}\sum_{x_0(q)}e\lp\ff_\gamma(x_0)\frac{a}{q}\rp\sum_{y\in\mathbb{Z}}   \int_{\mathbb{R}}\psi\lp\frac{x_0+qx}{X}\rp e\left(\ff_\gamma(x_0+qx)\beta\right)e(xy)dx\\
\nonumber=& \frac{X}{q}\sum_{\gamma\in B_T}\sum_{y\in\mathbb{Z}}\sum_{x_0(q)}e\lp\ff_\gamma(x_0)\frac{a}{q}\rp e\lp\frac{x_0y}{q}\rp \widehat{\psi}\lp A_\gamma X\beta+\frac{yX}{q}\rp e(B_\gamma\beta)\\
=&X\sum_{\gamma\in B_T}\sum_{y\in\mathbb{Z}}\bd{1}\{aA_\gamma-y\equiv 0(q)\}e\lp\frac{aB_\gamma}{q}\rp \widehat{\psi}\lp A_\gamma X\beta+\frac{yX}{q}\rp e(B_\gamma\beta)
}

As $\Psi$ is compactly supported, $\widehat{\Psi}$ decays faster than any polynomial rate.  Since $|\beta|<\frac{1}{qM}$, we have $|A_\gamma X\beta|<\frac{TX}{qM}\ll \frac{X}{qT^{\epsilon_0}}$.  Therefore, if $y\neq 0$, then $|A_\gamma X\beta +\frac{yX}{q}|\gg |\frac{yX}{q}|$.  Therefore, the contribution from $y\neq 0$ terms to \eqref{1102} is $\ll \sum_{y\neq 0} (\frac{q}{yX})^{L}\ll_{L} (\frac{M}{X})^{L}$ for any $L>1$, so is $O(N^{-L})$ for any $L\geq 1$, thus negligible.  We thus have
\al{\label{1015} \hr\lp\frac{a}{q}+\beta\rp= X\sum_{\gamma\in B_{T}}\bd{1}\{A_\gamma\equiv0(q)\}e\lp\frac{aB_\gamma}{q}\rp\widehat{\psi}\lp A_\gamma X \beta \rp e(B_\gamma\beta)+O(N^{-L}) 
}
for any $L\geq 1$.  \\

Since $\widehat{\Psi}(A_\gamma X\beta) \ll (N\beta)^{-L}$ for any $L\geq 1$, from \eqref{1015} we have
\al{\label{0800}\Big\vert \hr\lp\frac{a}{q}+\beta\rp \Big\vert \ll XT^{2\delta}(N\beta)^{-L}
}
for any $L\geq 1$, where we used $|A_\gamma|\gg T$.\\

Set $L=1$ in \eqref{0800} and apply it to the integral $I_1$.  We have
\al{{I}_1\ll \sum_{q\leq Q_0}\sum_{a(q)}{}^{'}\int_{-\frac{K_0}{N}}^{\frac{K_0}{N}}\lp\frac{N\beta}{K_0}\rp^2 \frac{X^2T^{4\delta}}{N^2\beta^2} d\beta \ll \frac{Q_0^2X^2T^{4\delta}}{NK_0}.
}
Thus if we set 
\al{ \label{0912}Q_0^2\ll K_0,}
then we have
$${I}_1\ll T^{4\delta-2}N^{1-\eta}.$$

For the integral $I_2$, again we apply \eqref{0800} and take $L$ large, and we have
$$I_2\ll \sum_{q\leq Q_0}\sum_{a(q)}{}^{'}\int_{|\beta|\geq \frac{K_0}{N}}X^2T^{4\delta}(N\beta)^{-L} d\beta \ll \frac{X^2T^{4\delta}}{NK_0^{L-1}}\ll T^{4\delta-2}N^{1-\eta}.$$\\

We deal the integral $I_3$ in the same way as $I_2$.  The integral $I_4$ is the most problematic, and our method requires $\delta$ close to 1 in order to get cancellation.  First we prove an auxiliary lemma:  \\
\begin{lem}\label{0141} Fix $\gamma\in B_T$, then we have 
$$\sum_{\gamma^{}{'}\in B_T}\bd{1}\{(A_{\gamma^{}{'}},B_{\gamma^{}{'}})=(A_\gamma,B_\gamma)\}\ll 1 $$
\end{lem}
\begin{proof}
Recalling the definition of $A_\gamma,B_\gamma$ from \eqref{0119}, the relation $(A_{\gamma^{}{'}},B_{\gamma^{}{'}})=(A_\gamma,B_\gamma)$ can be rephrased as 
\aln{\mat{0&v_1\\v_1&v_2}\mat{a_\gamma&b_\gamma\\c_\gamma&d_\gamma}\mat{w_1\\w_2}=\mat{0&v_1\\v_1&v_2}\mat{a_{\gamma^{}{'}}&b_{\gamma^{}{'}}\\c_{\gamma^{}{'}}&d_{\gamma^{}{'}}}\mat{w_1\\w_2}.
}
Cancelling the invertible matrix $\mat{0&v_1\\v_1&v_2}$ in the above relation, we see that ${\gamma^{}{'}}^{-1}\gamma$ stabilizes $\bd{w}$, or 
$$ \mat{a_{\gamma^{}{'}}&b_{\gamma^{}{'}}\\c_{\gamma^{}{'}}&d_{\gamma^{}{'}}}=\mat{a_\gamma&b_\gamma\\c_\gamma&d_\gamma} \mat{1-w_1w_2n&w_1^2n\\-w_2^2n&1+w_1w_2n} $$
for some $n\in \mathbb{Z}$.  
So we have 
\aln{c_{\gamma^{}{'}}=c_{\gamma}-\frac{w_2 A_{\gamma}}{v_1J}n.}

Since ${A}_\gamma\gg T$, $c_{\gamma}, c_{\gamma^{}{'}}\ll T$ and $v_1,w_2\neq 0$, we only have $\ll 1$ many choices for $n$, thus the lemma follows.  

\end{proof}

We split $I_4$ dyadically as follows:  write \al{\label{1019}I_Q=\sum_{Q<q<2Q}\sum_{a(q)}{}^{'}\int_{|\beta|\ll\frac{1}{N}}\Big\vert\hr(\frac{a}{q}+\beta)\Big\vert^2d\beta,}
where $Q_0<Q<M$.  It is enough to prove 
${I_Q}\ll T^{4\delta-2}N^{1-\eta}$ for every $Q\in[Q_0,M]$.\\

Using \eqref{1015}, we have 
\al{\label{0499}\nonumber I_Q\ll &\sum_{Q<q<2Q}\sum_{a(q)}{}^{'}\int_{|\beta|\ll \frac{1}{N}}X^2\sum_{\gamma\in B_{T}}\sum_{\gamma^{}{'}\in B_{T}}\bd{1}\{A_\gamma\equiv0(q)\}\bd{1}\{A_{\gamma^{}{'}}\equiv0(q)\}e\lp\frac{a(B_\gamma-B_{\gamma^{}{'}})}{q} \rp\\
\nonumber&\times\widehat{\psi}\lp A_\gamma X \beta \rp \widehat{\psi}\lp A_{\gamma^{}{'}} X \beta \rp e((B_\gamma-B_{\gamma^{}{'}})\beta)d\beta+O(N^{-L})\\
 \ll& \frac{X^2}{N}\sum_{Q<q<2Q}\sum_{\gamma\in B_{T}}\sum_{\gamma^{}{'}\in B_{T}}\bd{1}\{A_\gamma\equiv0(q)\}\bd{1}\{A_{\gamma^{}{'}}\equiv0(q)\}\vert c_q(B_\gamma-B_{\gamma^{}{'}})\vert+N^{-L}}
where $c_q(n)=\sum_{a(q)}'e(\frac{an}{q})$ is the Ramanujan's sum.  Fixing $n$, $c_q(n)$ is a multiplicative function with respect to $q$.  In the following, we will use the following elementary bound for $c_q(n)$: $$|c_q(n)|<\text{gcd}(q,n).$$ 

Apply Lemma \ref{0141} to \eqref{0499} and augment the set $\{(A_\gamma,B_\gamma):\gamma\in B_T\}$ to all vectors in a square $[-\alpha T,\alpha T]^2$ (we can take $\alpha$ to be, say $\sqrt{\left\langle\bd{v},\bd{v}\right\rangle_{\mathbb{R}^2}\cdot\left\langle\bd{w},\bd{w}\right\rangle_{\mathbb{R}^2}}$), 
\al{\label{0999}
I_Q\ll \frac{X^2}{N}\sum_{Q<q<2Q} \sum_{m,n,m',n'\ll T}\bd{1}\{m\equiv 0(q)\}\bd{1}\{m'\equiv0(q)\}\cdot\text{gcd}(q,n-n')+N^{-L}.
}

We split \eqref{0999} into $I_{Q}^{(=)}$ and $I_{Q}^{(\neq)}$ according to $n=n'$ or not.  For $I_{Q}^{(=)}$, we have 
\al{\nonumber
I_Q^{(=)}\ll &\frac{X^2}{N}\sum_{Q<q<2Q}\sum_{m,n,m'\ll T}  \bd{1}\{m\equiv0(q)\}\bd{1}\{m'\equiv0(q)\}\cdot q+N^{-L}\\
\ll &\frac{X^2}{N} \sum_{Q<q<2Q}  \frac{T^3}{q^2}q+N^{-L} \ll TN.
}
Since $\delta>5/6$, we have 
$$I_Q^{(=)}\ll TN\ll T^{4\delta-2}N^{1-\eta}.$$\\

Now we deal with $I_Q^{(\neq)}$: 
\al{\nonumber
I_Q^{(\neq)}\ll& \frac{X^2}{N}\sum_{\substack{n,n'\ll T\\n\neq n'}}\sum_{Q<q<2Q}\text{gcd}(q,n-n')\sum_{m,m'\ll T}\bd{1}\{m\equiv0(q)\}\bd{1}\{m'\equiv0(q)\}+N^{-L}\\
\nonumber\ll& \frac{X^2T^2}{NQ^2}\sum_{\substack{n,n'\ll T\\n\neq n'}}\sum_{t|(n-n{'})}t\sum_{\substack{Q<q<2Q\\t|q}}1+N^{-L}\\
\ll&\frac{NT^2}{Q}
}
Therefore we require \al{\label{0911}Q_0\gg T^{4-4\delta},} in order to make
$${I}_Q^{(\neq)}\ll T^{4\delta-2}N^{1-\eta}.$$\\

\noindent
$\bd{Conclusion:}$
Collecting \eqref{0917}, \eqref{0300}, \eqref{0912}, \eqref{0911}, we find we can set 
\al{\label{0933}K_0=T^{\frac{4\delta}{49}-\frac{10}{147}},  Q_0=T^{\frac{2\delta}{49}-\frac{5}{147}-\epsilon_1}, }
where $\epsilon_1$ is an arbitrary small positive number,  and we can take 
\al{\label{0936}\delta_0=\frac{593}{594}.}

\bibliographystyle{plain}
\bibliography{sl2}

\end{document}